\documentclass[]{article}

\usepackage[utf8]{inputenc}
\usepackage{amsmath}
\usepackage{amsfonts}
\usepackage{amssymb}
\usepackage{amsthm}
\usepackage{mathtools} 
\usepackage[margin=4cm]{geometry}
\usepackage{tkz-graph}
\usepackage{tikz}
\usepackage{hyperref}

\makeatletter
\newcommand*\rel@kern[1]{\kern#1\dimexpr\macc@kerna}
\newcommand*\widebar[1]{%
  \begingroup
  \def\mathaccent##1##2{%
    \rel@kern{0.8}%
    \overline{\rel@kern{-0.8}\macc@nucleus\rel@kern{0.2}}%
    \rel@kern{-0.2}%
  }%
  \macc@depth\@ne
  \let\math@bgroup\@empty \let\math@egroup\macc@set@skewchar
  \mathsurround\z@ \frozen@everymath{\mathgroup\macc@group\relax}%
  \macc@set@skewchar\relax
  \let\mathaccentV\macc@nested@a
  \macc@nested@a\relax111{#1}%
  \endgroup
}
\makeatother

\makeatletter
\def\blfootnote{\xdef\@thefnmark{}\@footnotetext}
\makeatother

\newtheorem{theorem}{Theorem}

\newtheorem{lemma}[theorem]{Lemma}
\newtheorem{proposition}[theorem]{Proposition}
\newtheorem{example}[theorem]{Example}

\newcommand{\trans}{\mathsf{T}} 
\newcommand{\Zk}{Z_k} 
\newcommand{\R}{{\mathbb R}}

\newcommand{\dd}[2]{\frac{\mathrm{d} #1}{\mathrm{d} #2}}
\newcommand{\DD}[2]{\frac{\partial #1}{\partial #2}}

\DeclareMathOperator{\diag}{diag}

\DeclareMathOperator{\im}{im}
\DeclareMathOperator{\sign}{sign}

\begin{document}

\title{Complex-balanced equilibria \\ of generalized mass-action systems: \\
Necessary conditions for linear stability}

\author{Bal\'azs Boros, Stefan M\"uller$^*$, and Georg Regensburger}
\blfootnote{
\scriptsize

\noindent
{\bf B.~Boros} (\href{mailto:balazs.boros@univie.ac.at}{balazs.boros@univie.ac.at}),
{\bf S.~M\"uller} (\href{mailto:st.mueller@univie.ac.at}{st.mueller@univie.ac.at}),
Faculty of Mathematics, University of Vienna, Oskar-Morgenstern-Platz 1, 1090 Wien, Austria

\smallskip
\noindent
{\bf G.~Regensburger} (\href{mailto:georg.regensburger@jku.at}{georg.regensburger@jku.at}),
Institute for Algebra, Johannes Kepler University Linz,
Altenbergerstra{\ss}e 69, 4040 Linz, Austria.

\smallskip
\noindent
$^*$ Corresponding author.
}

\maketitle

\begin{abstract}
It is well known that, for mass-action systems, complex-balanced equilibria are asymptotically stable.
For {\em generalized} mass-action systems, even if there exists a unique complex-balanced equilibrium (in every stoichiometric class and for all rate constants), it need not be stable.

We first discuss several notions of {\em matrix stability (on a linear subspace)} such as D-stability and diagonal stability,
and then we apply our abstract results to complex-balanced equilibria of generalized mass-action systems.
In particular, we show that linear stability (on the stoichiometric subspace and for all rate constants) {\em implies} uniqueness.
For cyclic networks,
we {\em characterize} linear stability 
(in terms of D-stability of the Jacobian matrix); 
and for weakly reversible networks,
we give {\em necessary} conditions for linear stability (in terms of D-semistability of the Jacobian matrices of all cycles in the network).
Moreover, we show that, for {\em classical} mass-action systems,
complex-balanced equilibria are not just asymptotically stable, but even diagonally stable (and hence linearly stable).

Finally, we recall and extend characterizations of D-stability and diagonal stability for matrices of dimension up to three,
and we illustrate our results by examples of irreversible cycles (of dimension up to three) and of reversible chains and S-systems (of arbitrary dimension).
\end{abstract}


\section{Introduction}

In their foundational paper from 1972,
Horn and Jackson considered chemical reaction networks (CRNs) with mass-action kinetics~\cite{HornJackson1972}.
In particular, they proved that complex-balanced equilibria are asymptotically stable
(for all rate constants),
by using an entropy-like Lyapunov function.
Moreover, they observed that every CRN with power-law kinetics can be written as another CRN with mass-action kinetics
(possibly with non-integer stoichiometric coefficients).
Typically, the resulting network
does not have desired properties such as weak reversibility and zero deficiency.
The more recent definition of CRNs with generalized mass-action kinetics
(involving both stoichiometric coefficients and kinetic orders)
allows to study power-law kinetics without having to rewrite the network~\cite{MuellerRegensburger2012,MuellerRegensburger2014}.

For the resulting {\em generalized mass-action systems}, 
existence and uniqueness of complex-balanced equilibria (in every stoichiometric class and for all rate constants) is well understood~\cite{MuellerRegensburger2012,MuellerHofbauerRegensburger2019}, 
but not much is known about their stability.
Even planar S-systems with a unique complex-balanced equilibrium
display rich dynamical behavior,
including super/sub-critical or degenerate Hopf bifurcations, centers, and up to three limit cycles;
see~\cite{BorosHofbauerMueller2017,BorosHofbauerMuellerRegensburger2018, BorosHofbauerMuellerRegensburger2019,BorosHofbauer2019}.

Since it is hard to find Lyapunov functions for generalized mass-action systems,
we approach the problem by linearization.
In other words, instead of asymptotic stability,
we investigate linear stability (on the stoichiometric subspace and for all rate constants).
To this end, we first discuss several notions of matrix stability (on a linear subspace)
such as D-stability and diagonal stability; 
see~\cite[Ch.~2]{HornJohnson1991},~\cite[Ch.~26]{Hogben2014}, and the references therein.
In particular, diagonal stability of the Jacobian matrix allows to construct Lyapunov functions with separated variables.
Our main results characterize linear stability of complex-balanced equilibria
(on the stoichiometric subspace and for all rate constants) for {\em cyclic} networks and give necessary conditions for {\em weakly reversible} networks.

In the setting of {\em classical} mass-action systems,
we prove that complex-balanced equilibria are not just asymptotically stable, but even diagonally stable (and hence linearly stable). For an alternative proof, see~\cite[Thm.~15.2.2.]{Feinberg2019}.
As opposed to asymptotic stability, linear stability is robust with respect to small perturbations of the system.
In particular, this allows to show the robustness of the classical deficiency zero theorem
with respect to small perturbations of the kinetic orders from the stoichiometric coeffcients;
see~\cite[Cor.~47]{MuellerHofbauerRegensburger2019}.


{\bf Organization of the work and main results.}
In Section~\ref{sec:gmas},
we introduce generalized mass-action systems,
and in Section~\ref{sec:stability_notions},
we discuss several notions of matrix stability (on a linear subspace).
In Section~\ref{sec:main}, we present our main results:
\begin{itemize}
\item[4.1] For classical mass-action systems, 
complex-balanced equilibria are diagonally stable (and hence linearly stable); see Theorem~\ref{prop:YequalsYtilde}.
\end{itemize}
\begin{itemize}
\item[4.2] For generalized mass-action systems, linear stability of complex balanced equilibria
implies uniqueness; see Theorem~\ref{thm:main_injectivity}.
\item[4.3] For cyclic networks,
linear stability is equivalent to D-stability of the Jacobian matrix; see Theorem~\ref{thm:main_cycle}.
\item[4.4] For weakly reversible networks,
linear stability implies D-semistability of the Jacobian matrices of all cycles in the network; see Theorem~\ref{thm:main_wr}.
\end{itemize}
In Section~\ref{sec:characterize_restricted_stabilities}, 
we recall and extend characterizations of notions of matrix stability,
and finally, in Section~\ref{sec:examples},
we illustrate our results by a series of examples.




\section{Generalized mass-action systems} \label{sec:gmas}

A {\em generalized chemical reaction network} $(G,y,\tilde y)$ is based on a directed graph $G=(V,E)$ without self-loops;
every vertex $i \in V=\{1,\ldots,m\}$ is labeled with a {\em (stoichiometric) complex} $y(i) \in \R^n_{\ge0}$, 
and every source vertex $i \in V_s \subseteq V$ is labeled with a {\em kinetic-order complex} $\tilde y(i) \in \R^n$.
If every component of $G$ is strongly connected, we call $G$ and $(G,y,\tilde y)$ {\em weakly reversible}.

A {\em generalized mass-action system} $(G_k,y,\tilde y)$ is a generalized chemical reaction network $(G,y,\tilde y)$
together with edge labels $k \in \R^E_+$, resulting in the labeled digraph $G_k$.
Every edge $(i \to i') \in E$, representing the {\em chemical reaction} $y(i) \to y(i')$, is labeled with a {\em rate constant} $k_{i \to i'} > 0$. 

The ODE system for the {\em species concentrations} $x \in \R^n_+$,
associated with the generalized mass-action system $(G_k,y,\tilde y)$, is given by
\begin{equation} \label{ode1}
\dd{x}{t} = \sum_{(i \to i') \in E} k_{i \to i'} \, x^{\tilde y(i)} \big(y(i')-y(i)\big) .
\end{equation}
The sum ranges over all reactions, 
and every summand is a product of the reaction rate and the difference of product and educt complexes.
Thereby, for $x \in \R^n_+$ and $y \in \R^n$, we define $x^y = x_1^{y_1} \cdots x_n^{y_n} \in \R_+$.
Moreover, for $Y = [y^1, \ldots, y^m] \in \R^{n \times m}$, we define $x^Y \in \R^m_+$ via $(x^Y)_j = x^{y^j}$ for $j = 1, \ldots, m$.

The right-hand-side of the ODE system~\eqref{ode1} can be decomposed into stoichiometric, graphical, and kinetic-order contributions,
\begin{equation} \label{ode2}
\dd{x}{t} = Y I_E \diag(k) (I_E^s)^\trans \, x^{\widetilde Y} = Y A_k \, x^{\widetilde Y} ,
\end{equation}
where $Y,\, \widetilde Y \in \R^{n \times V}$ are the matrices of stoichiometric and kinetic-order complexes,
$I_E, I_E^s \in \R^{V \times E}$ are the incidence and source matrices%
\footnote{%
Explicitly, $(I_E)_{i,j \to j'} = -1$ if $i=j$, $(I_E)_{i,j \to j'} = 1$ if $i=j'$, and $(I_E)_{i,j \to j'} = 0$ otherwise.
Further, $(I_E^s)_{i,j \to j'} = 1$ if $i=j$ and $(I_E^s)_{i,j \to j'} = 0$ otherwise.
Finally, $(A_k)_{i',i} = k_{i \to i'}$ if $(i \to i') \in E$, $(A_k)_{i,i} = - \sum_{i \to i'} k_{i \to i'}$, and $(A_k)_{i',i} = 0$ otherwise.
}
of the digraph $G$, 
and $A_k = I_E \diag(k) (I_E^s)^\trans \in \R^{V \times V}$ is the Laplacian matrix of the labeled digraph $G_k$.
(We note that columns of $\widetilde Y$ corresponding to non-source vertices can be chosen arbitrarily.)

Clearly, the change over time lies in the {\em stoichiometric subspace}
\[
S = \im (Y I_E) ,
\]
that is, $\dd{x}{t} \in S$.
Equivalently, trajectories are confined to cosets of $S$,
that is, $x(t) \in x(0)+S$.
For $x' \in \R^n_+$, the set $(x'+S) \cap \R^n_+$ is called a {\em stoichiometric class}.

Analogously, we introduce the {\em kinetic-order subspace}
\[
\widetilde S = \im (\widetilde Y I_E) .
\]


A positive equilibrium $x \in \R^n_+$ of the ODE system~\eqref{ode2} that fulfills
$
A_k \, x^{\widetilde Y} = 0
$
is called a {\em complex-balanced equilibrium}.
The set of all complex-balanced equilibria is given by
\[
Z_k = \{ x \in \R^n_+ \mid A_k \, x^{\widetilde Y} = 0 \} .
\] 

The Jacobian matrix of the right-hand-side of~\eqref{ode2} is given by
\begin{equation} \label{jac}
J(x) = Y A_k \diag(x^{\widetilde Y}) \, \widetilde Y^\trans \diag(x^{-1}) .
\end{equation}

For given generalized chemical reaction network $(G,y,\tilde y)$,
we study whether, for all rate constants $k\in\R^E_+$ (such that $\Zk\neq\emptyset$),
all complex-balanced equilibria $x^* \in \Zk$ are linearly stable on $S$,
that is, the corresponding Jacobian matrices $J(x^*)$ are stable on $S$.


Before we turn to stability,
we recall the well-known fact that every positive vector is a complex-balanced equilibrium for some rate constant
(see e.g.\ the proof of Lemma~1 in~\cite{MuellerRegensburger2014}).
Further, we state a characterization of the uniqueness of complex-balanced equilibria
in terms of sign vectors of the stoichiometric and kinetic-order subspaces
(see Proposition~3.1 in~\cite{MuellerRegensburger2012}).

\begin{proposition} \label{pro:k}
Consider a weakly reversible generalized chemical reaction network, and let $x^* \in \R_+^n$. 
Then, there exist rate constants $k \in \R^E_+$ such that $x^* \in Z_k$.
\end{proposition}



\begin{proposition}\label{pro:injectivity}
Consider a weakly reversible generalized chemical reaction network.
There exists at most one complex-balanced equilibrium in every stoichiometric class 
and for all rate constants if and only if $\sign(S) \cap \sign(\widetilde S^\perp) \neq \{0\}$.
\end{proposition}

For surveys on the uniqueness of equilibria and related injectivity results,
see~\cite{MuellerFeliuRegensburgerEtAl2016, BanajiPantea2016, FeliuMuellerRegensburger2019}.

%
%
%


\section{Notions of matrix stability} \label{sec:stability_notions}

Let $S$ be a linear subspace of $\R^n$.
For an ODE system $\dot x = f(x)$ with $f \colon \R^n \to \R^n$ and $\im f \subseteq S$, 
cosets of $S$ are forward invariant.
Hence, given an equilibrium $x \in \R^n$,
one is interested in its asymptotic stability on the coset $x+S$.
To approach the problem via linearization, one considers the Jacobian matrix $J(x) = \DD{f}{x} \in \R^{n \times n}$,
more precisely, the linear map $J(x)|_S  \colon  S \to S$.
By the Hartman-Grobman Theorem, 
if $x$ is hyperbolic (that is, if all eigenvalues  of $J(x)|_S$ have non-zero real part),
then the original ODE system is dynamically equivalent (technically: topologically conjugate) 
to the linear ODE system $\dot y = J(x) y$ on $S$; see e.g.~\cite[Thm.~9.9]{Teschl2012}.
In the following,
we recall notions of stability of a square matrix and extend them to stability on a linear subspace.

A square matrix is \emph{stable} (respectively, \emph{semistable}) if all its eigenvalues have negative (respectively, non-positive) real part. We start with a matrix formulation of Lyapunov's Theorem; see~\cite[Ch.~XV, Thm.~3']{Gantmacher1959} or~\cite[Thm.~2.2.1]{HornJohnson1991}. 

\begin{proposition} \label{pro:lyapunov_classical}
Let $A\in\R^{n \times n}$. 
The following implications hold:
\begin{center}
\begin{tabular}{ccc}
\begin{tabular}{c}
There exists $P=P^\trans>0$ \\
s.t.\ $PA+A^\trans P<0$.
\end{tabular} & $\Rightarrow$ & 
\begin{tabular}{c}
There exists $P=P^\trans>0$ \\
s.t.\ $PA+A^\trans P\leq0$.
\end{tabular} \\
\\[-1ex]
$\Updownarrow$ & & $\Downarrow$ \\
\\[-1ex]
\begin{tabular}{c}
$A$ is stable.
\end{tabular}  & $\Rightarrow$ & 
\begin{tabular}{c}
$A$ is semistable.
\end{tabular}
\end{tabular}
\end{center}
\end{proposition}

\newpage

\begin{proof}
Obviously, the implications from left to right hold.

The equivalence on the left is Lyapunov's Theorem.

Finally, if there exists $P=P^\trans>0$ with $PA+A^\trans P\leq 0$, 
then the origin is Lyapunov stable for the linear ODE $\dot x = Ax$ (by the Lyapunov function $x \mapsto x^\trans P x$), 
and thus $A$ cannot have an eigenvalue with positive real part.
\end{proof}

Let $S$ be a linear subspace.
We say that a square matrix $A$ with $\im A \subseteq S$ is \emph{stable on $S$} (respectively, \emph{semistable on $S$}) if all eigenvalues of the linear map $A|_S\colon S\to S$ have negative (respectively, non-positive) real part. We extend Lyapunov's Theorem to stability on a linear subspace.

\begin{proposition}\label{pro:lyap_on_S}
Let $A \in \R^{n \times n}$ and $S \subseteq \R^n$ be a linear subspace with $\im A \subseteq S$. 
The following implications hold:
\begin{center}
\begin{tabular}{ccc}
\begin{tabular}{c}
There exists $P=P^\trans>0$ on $S$ \\ 
s.t.\ $PA+A^\trans P<0$ on $S$.
\end{tabular} & $\Rightarrow$ & 
\begin{tabular}{c}
There exists $P=P^\trans>0$ on $S$ \\
s.t.\ $PA+A^\trans P\leq0$ on $S$.
\end{tabular} \\
\\[-1ex]
$\Updownarrow$ & & $\Downarrow$ \\
\\[-1ex]
\begin{tabular}{c} $A$ is stable on $S$.\end{tabular} & $\Rightarrow$ & 
\begin{tabular}{c} $A$ is semistable on $S$.\end{tabular}
\end{tabular}
\end{center}
\end{proposition}
\begin{proof}
Obviously, the implications from left to right hold.

To prove the equivalence on the left, 
let $s = \dim S$ and $S = \im B$ with $B \in \R^{n\times s}$ and $B^\trans B = I \in \R^{s \times s}$. 
Then, for every $x\in S$, there exists a unique $y\in\R^s$ such that $x = By$. 
Further, $y=B^\trans x$ and, using $\dot x = Ax$ (on $S$), we have $\dot y = B^\trans A B y$ (on $\R^s$). 
Thus, $A$ is stable on $S$ if and only if $B^\trans A B \in \R^{s \times s}$ is stable on $\R^s$. 
By Theorem~\ref{pro:lyapunov_classical}, the latter is equivalent to the existence of $Q \in \R^{s\times s}$ with $Q=Q^\trans>0$ such that $y^\trans(Q B^\trans A B + B^\trans A^\trans BQ) y<0$ for all $0\neq y \in \R^s$. 
This is further equivalent to the existence of $P = B Q B^\trans \in \R^{n \times n}$ with $P = P^\trans > 0$ on $S$
such that $x^\trans (PA+A^\trans P) x<0$ for all $0\neq x \in S$.

Finally, 
let $P=P^\trans>0$ on $S$ and $PA+A^\trans P\leq0$ on $S$. 
With $B$ as above, let $Q = B^\trans P B \in \R^{s \times s}$
and note that $A = B B^\trans A$ (because $x = B B^\trans x$ for all $x \in S$).
Then $Q=Q^\trans>0$ and $QB^\trans A B + B^\trans A^\trans B Q \le 0$.
By Proposition~\ref{pro:lyapunov_classical}, $B^\trans A B$ is semistable, and, as in the previous paragraph, 
the latter is equivalent to $A$ being semistable on $S$.
\end{proof}

We recall more notions of stability of square matrices. 
For convenience, let $\mathcal{D}_+ \subseteq \R^{n \times n}$ denote the set of diagonal matrices with positive diagonal. 
A matrix $A \in \R^{n\times n}$ is \emph{diagonally stable} (respectively, \emph{diagonally semistable}) if there exists $P \in \mathcal{D}_+$ such that $PA+A^\trans P<0$ (respectively, $PA+A^\trans P\leq0$). 
We note that diagonal stability is also known as Lyapunov diagonal stability or Volterra-Lyapunov stability.
A matrix $A \in \R^{n\times n}$ is \emph{D-stable} (respectively, \emph{D-semistable}) if $AD$ is stable (respectively, semistable) for all $D \in \mathcal{D}_+$.
The following result summarizes the relations between these notions.

\newpage

\begin{proposition} \label{pro:stab_implications_classical}
Let $A\in\R^{n \times n}$. 
The following implications hold:
\begin{center}
\begin{tabular}{ccc}
A is diagonally stable.
& $\Rightarrow$ &
$A$ is diagonally semistable. \\
\\[-1ex]
$\Downarrow$ &  & $\Downarrow$ \\
\\[-1ex]
A is D-stable.
& $\Rightarrow$ &
$A$ is D-semistable. \\
\\[-1ex]
$\Downarrow$ &  & $\Downarrow$ \\
\\[-1ex]
$A$ is stable. & $\Rightarrow$ &
$A$ is semistable.
\end{tabular}
\end{center}
%
\end{proposition}
\begin{proof}
Obviously, the implications from left to right hold.

To prove the implications from the first to the second row,
let $P \in \mathcal{D}_+$ be such that $PA+A^\trans P<0$ (respectively, $PA+A^\trans P\leq0$).
Then, for any $D \in \mathcal{D}_+$, we have $(DP)(AD)+(AD)^\trans(DP)= D(PA+A^\trans P)D < 0$ (respectively, $\le 0$), and by Proposition~\ref{pro:lyapunov_classical}, $AD$ is stable (respectively, semistable).

Finally, the implications from the second row to the third row are trivial.
(If $AD$ is (semi)stable for any $D\in\mathcal{D}_+$, then this holds for $D=I$.)
\end{proof}


Let $S$ be a linear subspace.
We say that $A \in \R^{n \times n}$ with $\im A \subseteq S$ is \emph{diagonally stable on $S$} (respectively, \emph{diagonally semistable on $S$}) 
if there exists $P \in \mathcal{D}_+$ such that $PA+A^\trans P<0$ on $S$ (respectively, $PA+A^\trans P\le0$ on $S$).
We say that $A \in \R^{n \times n}$ with $\im A \subseteq S$ is \emph{D-stable on $S$} (respectively, \emph{D-semistable on $S$}) 
if $AD$ is stable on $S$ (respectively, semistable on $S$) for all $D \in \mathcal{D}_+$.
Finally, we introduce an even stronger notion.
We say that $A \in \R^{n \times n}$ with $\im A \subseteq S$ is \emph{diagonally D-stable on $S$} (respectively, \emph{diagonally D-semistable on $S$})
if, for all $D \in \mathcal{D}_+$, there exists $P \in \mathcal{D}_+$ such that $PAD+DA^\trans P<0$ on $S$ (respectively, $PAD+DA^\trans P\le0$ on $S$).
We note that diagonal D-stability on $S$ trivially implies diagonal stability on $S$,
and the two notions agree for $S=\R^n$, see also~\cite[p.~257]{Cross1978}.
Moreover, we have the following relations.

\begin{proposition} \label{pro:stab_implications_on_S}
Let $A \in \R^{n \times n}$ and $S \subseteq \R^n$ be a linear subspace with $\im A \subseteq S$. 
The following implications hold:
\begin{center}
\begin{tabular}{ccc}
A is diagonally D-stable on $S$.
& $\Rightarrow$ &
$A$ is diagonally D-semistable on $S$. \\
\\[-1ex]
$\Downarrow$ &  & $\Downarrow$ \\
\\[-1ex]
A is D-stable on $S$.
& $\Rightarrow$ &
$A$ is D-semistable on $S$. \\
\\[-1ex]
$\Downarrow$ &  & $\Downarrow$ \\
\\[-1ex]
$A$ is stable on $S$. 
& $\Rightarrow$ &
$A$ is semistable on $S$.
\end{tabular}
\end{center}
%
\end{proposition}
\begin{proof}
Obviously, the implications from left to right hold.

The implications from the first to the second row follow immediately from Proposition~\ref{pro:lyap_on_S}.

Finally, the implications from the second row to the third row are trivial.
\end{proof}



\section{Main results} \label{sec:main}


In the following, we say that an equilibrium $x^*$ is {\em linearly stable} (in its stoichiometric class $x^*+S$)
if the Jacobian matrix $J(x^*)$ is stable on $S$.
In short:
\begin{center}
$x^*$ is linearly stable (in $x^* + S$), if $J(x^*)$ is stable on $S$.
\end{center}    

Analogously, we say that $x^*$ is diagonally D-stable/diagonally stable/D-stable (in $x^*+S$)
if $J(x^*)$ is diagonally D-stable/diagonally stable/D-stable on $S$.


\subsection{Linear stability for mass-action systems}

We start by showing that,
for {\em classical} mass-action systems (where $\tilde{y}=y$),
complex-balanced equilibria are diagonally D-stable (and hence diagonally stable/D-stable/linearly stable) in their stoichiometric classes,
and not just locally asymptotically stable (as shown via the classical entropy-like Lyapunov function).

Recently, linear stability of complex-balanced equilibria was shown in~\cite[Theorem~15.2.2.]{Feinberg2019}.
In fact, the author proved {\em diagonal stability},
without noting it and using an ad-hoc inequality.
Here, we even show {\em diagonal D-stability},
thereby using the negative semi-definiteness of the Laplacian matrix of an undirected graph.

\begin{theorem} \label{prop:YequalsYtilde}
Consider a weakly reversible chemical reaction network. 
Then, for all rate constants, complex-balanced equilibria are linearly stable (in their stoichiometric classes).
In fact, they are diagonally D-stable (and hence also diagonally stable, D-stable, and linearly stable).
\end{theorem}
\begin{proof}
Let $k \in \R^E_+$ and $x^* \in \Zk$. We show that the corresponding Jacobian matrix $J=J(x^*)$ is diagonally D-stable on $S$. By Proposition~\ref{pro:stab_implications_on_S}, all other conclusions follow.

Let $D \in \mathcal{D}_+$. 
We show that there exists $P = \diag((x^*)^{-1})D \in \mathcal{D}_+$ 
such that $H = P J D + DJ^\trans P<0$ on $S$. 
Let $k^* \in \R^E_+$ be defined by $k^*_{i \to i'}=k_{i \to i'} \, (x^*)^{y(i)}$ for $(i \to i') \in E$. 
Then, $A_k \diag((x^*)^Y) = A_{k^*}$.
Using~\eqref{jac} for the Jacobian matrix $J$, we have
\[
H = P \, Y \left(A_{k^*} + A_{k^*}^\trans \right)Y^\trans P.
\]
Now, let $A_{\bar k} = A_{k^*} + A_{k^*}^\trans$ and hence $H = P \, Y A_{\bar k} \, Y^\trans P$. 
The symmetric matrix $A_{\bar k}$ is the Laplacian matrix of a labeled undirected graph $\widebar G_{\bar k}$ with $\widebar G=(V,\widebar E)$ and $\bar k \in \R^{\widebar E}_+$. 
In particular, $A_{\bar{k}} = -I_{E'} \diag(\bar k) \, I_{E'}^\trans$ for some directed version $E'$ of the undirected edges $\widebar E$ and the corresponding incidence matrix $I_{E'}$ (with $\im I_{E'} = \im I_E$);
see e.g.~\cite{BrualdiRyser1991}.
Hence, $A_{\bar k} \le 0$ and $H = - P \, Y  I_{E'} \diag(\bar k) \, I_{E'}^\trans \, Y^\trans P \le 0$ on $S$. 

Suppose $v^\trans Hv=0$ for some $v \in S$.
Then, $I_{E'}^\trans \, Y^\trans Pv = 0$.
Now, $S =  \im (Y I_E) = \im (Y I_{E'})$, $S^\perp = \ker (I_{E'}^\trans \, Y^\trans)$, and hence $Pv \in S^\perp$.
Clearly, $v^\trans P v = 0$ which finally implies $v=0$, since $P>0$.

%

Hence, $H<0$ on $S$, 
and $J$ is diagonally D-stable on $S$.
\end{proof}

In the result above, we prove {\em diagonal stability} of $J$ on $S$, 
by providing a diagonal, positive definite matrix $P$ such that $PJ+J^\trans P < 0$ on $S$.
This property implies the existence of a function with separated variables $L(x) = \sum_{i=1}^nL_i(x_i)$ 
that serves as a Lyapunov function for showing the asymptotic stability of the complex-balanced equilibrium $x^*\in \R^n_+$ (in its stoichiometric class). 


As an example, consider a scaled version of the classical entropy-like Lyapunov function
$L(x) = \sum_{i=1}^n p_i \, x_i^* \, [x_i(\log(x_i/x_i^*)-1)+x_i^*]$, where $P = \diag(p_1,\ldots,p_n)$,
and the function $g(x) = \nabla L(x) f(x)$, where $\dd{x}{t} = f(x)$.
Then, the gradient of $g$ at $x^*$ vanishes, and the Hessian matrix of $g$ at $x^*$ is $H = PJ+J^\trans P$ with $H < 0$ on $S$.
Hence, in $x^* + S$, the function $g$ has a local maximum at $x^*$ with $g(x^*)=0$. 
That is, $L$ is a Lyapunov function at $x^*$, and $x^*$ is asymptotically stable (in $x^* + S$).


\subsection{Linear stability implies uniqueness}

Now, we turn to {\em generalized} mass-action systems.
We show that linear stability of complex-balanced equilibria implies uniqueness (in their stoichiometric classes).
We start with a useful technical result.

\begin{lemma} \label{lem:kerLaplacian}
Consider a weakly reversible generalized mass-action system, and let $x^* \in Z_k$. 
Then, $\ker (A_k \diag((x^*)^{\widetilde Y})) = \ker I_E^\trans$.
\end{lemma}
\begin{proof}
The dimension of $\ker I_E^\trans$ equals the number of connected components of the graph $G=(V,E)$,
and the characteristic vectors of the vertex sets of the connected components form a basis.

Let $k^* \in \R^E_+$ be defined by $k^*_{i \to i'} = k_{i \to i'} \, (x^*)^{\tilde{y}(i)}$ for $(i \to i') \in E$. 
Then, $A_k \diag((x^*)^{\widetilde Y}) = A_{k^*}$. 
Clearly, $A_{k^*}$ is the Laplacian matrix of the labeled directed graph $G_{k^*}$, 
and the dimension of its kernel also equals the number of connected components of $G$;
see e.g.~\cite[Section~4]{MuellerRegensburger2014} and the references therein.

By definition, $x^* \in Z_k$ if $A_k (x^*)^{\widetilde Y} = 0$. 
Due to the block structure of the Laplacian matrix, 
the characteristic vectors of the vertex sets of the connected components are in the kernel of $A_k \diag((x^*)^{\widetilde Y})$.
(As an example, consider a strongly connected graph $G$, 
and let $1^V$ be the characteristic vector of $V$.
Then, $A_k \diag((x^*)^{\widetilde Y}) \, 1^V = A_k (x^*)^{\widetilde Y} = 0$.)
Altogether, $\ker (A_k \diag((x^*)^{\widetilde Y})) = \ker I_E^\trans$.
\end{proof}

\begin{theorem} \label{thm:main_injectivity}
Consider a weakly reversible generalized chemical reaction network. 
If, for all rate constants, complex-balanced equilibria are linearly stable (in their stoichiometric classes), 
then they are unique (in their stoichiometric classes).
\end{theorem}

\begin{proof}
Assume that, for some $k' \in \R^E_+$, there exists a stoichiometric class with at least two complex-balanced equilibria. 
Then, by Proposition~\ref{pro:injectivity}, there exist vectors $0\neq u \in \widetilde S^\perp$ and $0\neq v \in S$ with $\sign(u) = \sign(v)$. 
Now, $\widetilde S = \im (\widetilde Y I_E)$, $\widetilde S^\perp = \ker (I_E^\trans \, \widetilde Y^\trans)$, 
and hence $I_E^\trans \, \widetilde Y^\trans u = 0$.
Let $x^*\in\R^n_+$ be such that $u = \diag((x^*)^{-1}) v$, and let $k\in\R^E_+$ be such that $x^*\in Z_k$, see Proposition~\ref{pro:k}. 
Then, 
\[
I_E^\trans \, \widetilde Y^\trans \diag((x^*)^{-1}) v = 0 ,
\]
that is, $\widetilde Y^\trans \diag((x^*)^{-1}) v \in \ker I_E^\trans = \ker (A_k \diag((x^*)^{\widetilde Y}))$,
the latter by Lemma~\ref{lem:kerLaplacian}.
Using~\eqref{jac} for the Jacobian matrix $J$, we have
\begin{align*}
Jv = Y \underbrace{A_k \diag((x^*)^{\widetilde Y})} \, \underbrace{\widetilde Y^\trans \diag((x^*)^{-1}) v} = 0 .
\end{align*}
Hence, for all $P=P^\trans>0$, we have $v^\trans(PJ+J^\trans P)v=0$, 
and by Proposition~\ref{pro:lyap_on_S}, $J$ is not stable on $S$. 
\end{proof}

As can be demonstrated by the generalized reaction network $\mathsf{X} (\mathsf{X}) \rightleftharpoons \mathsf{Y} (\mathsf{0})$, 
linear stability implies uniqueness, but not existence in every stoichiometric class.


\subsection{The network is a cycle} \label{subsec:main_cycle}

In the following result, we characterize diagonal stability and linear stability of complex-balanced equilibria, 
provided that the network is a cycle.

\begin{theorem} \label{thm:main_cycle}
Consider a \emph{cyclic} generalized chemical reaction network,
and let $A = Y A_{k=1} \widetilde{Y}^\trans$.
The following implications hold:
\begin{center}
\begin{tabular}{ccc} 
\begin{tabular}{c}
For all rate constants, \\ complex-balanced equilibria \\ are diagonally stable \\ (in their stoichiometric classes).
\end{tabular} 
& $\Rightarrow$ &
\begin{tabular}{c}
For all rate constants, \\ complex-balanced equilibria \\ are linearly stable \\ (in their stoichiometric classes).
\end{tabular} \\
\\[-1ex]
$\Updownarrow$ & & $\Updownarrow$  \\
\\[-1ex]
\begin{tabular}{c}
$A$ is diagonally D-stable on $S$.
\end{tabular}
& $\Rightarrow$ &
\begin{tabular}{c}
$A$ is D-stable on $S$.
\end{tabular}
\end{tabular}
\end{center}

\end{theorem}
\begin{proof}
The implication in the first (respectively, second) row follows immediately from Proposition~\ref{pro:lyap_on_S} (respectively, Proposition~\ref{pro:stab_implications_on_S}).

To prove the two equivalences,
let $k \in \R^E_+$ and $x^* \in \Zk$. 
Since $G=(V,E)$ is a cycle, the quantity $c = k_{i \to i'} \, (x^*)^{\tilde{y}(i)}$ is the same for all $(i \to i')\in E$.
In particular, $A_k \diag((x^*)^{\widetilde{Y}}) = c \, A_{k=1}$.
Using~\eqref{jac} for the Jacobian matrix $J$, we have
\[
J = c \, Y A_{k=1} \widetilde{Y}^\trans D
\]
with $D = \diag((x^*)^{-1})$.
Finally, recall that every $x^*\in\R^n_+$ is a complex-balanced equilibrium for some $k\in\R^E_+$,
see Proposition~\ref{pro:k}.
Hence, every $D \in \mathcal{D}_+$ is of the form $D = \diag((x^*)^{-1})$ for some $k$,
and the equivalences follow.
\end{proof}


\subsection{The network is weakly reversible} \label{subsec:main_wr}

In the previous subsection, we characterized diagonal stability and linear stability of complex-balanced equilibria, provided that the network is a cycle. In this subsection, we extend those results to weakly reversible networks, but instead of equivalent we only give necessary conditions for diagonal and linear stability. 
We start with a useful technical result.

\begin{lemma} \label{lem:approx_a_cycle}
Consider a weakly reversible generalized chemical reaction network, let $C$ be a cycle of the network, and let $x^* \in \R_+^n$. Then, there exists a family of rate constants $k^\varepsilon \in \R^E_+$ (with $\varepsilon>0$) such that $x^* \in Z_{k^\varepsilon}$ for every $\varepsilon>0$ and $J^\varepsilon \to Y A_{k=1}^C \widetilde{Y}^\trans \diag((x^*)^{-1})$ as $\varepsilon \to 0$, where $A_{k=1}^C$ is the Laplacian matrix of the cycle $C$ with all rate constants set to $1$.
\end{lemma}
\begin{proof}
For a cycle $C'$, let $k^{C'} \in \R^E_{\geq0}$ be defined by
\begin{align*}
k^{C'}_{i \to i'} = \frac{1}{(x^*)^{\tilde{y}(i)}}\cdot
\begin{cases}
1, & \text{ if } (i \to i') \in C', \\
0,           & \text{ if } (i \to i') \notin C' .
\end{cases}
\end{align*}
For $\varepsilon>0$, let $k^\varepsilon = k^C + \varepsilon \sum_{C'\neq C} k^{C'}$, 
where the summation is over all cycles $C'$ of $G$, except for the given cycle $C$. 
Then, $x^* \in Z_{k^\varepsilon}$ and
\[
A_{k^\varepsilon} \diag((x^*)^{\widetilde Y}) = A_{k=1}^C+\varepsilon\sum_{C'\neq C}A_{k=1}^{C'}.
\]
Using~\eqref{jac} for the Jacobian matrix~$J^\varepsilon$, 
we obtain the desired limit as $\varepsilon \to 0$.
\end{proof}

\begin{theorem} \label{thm:main_wr}
Consider a weakly reversible generalized chemical reaction network,
and, for a cycle $C$ of the network, let $A^C = Y A^C_{k=1} \widetilde{Y}^\trans$
and $S^C$ be the corresponding stoichiometric subspace.
The following implications hold:
\begin{center}
\begin{tabular}{ccc}
\begin{tabular}{c}
For all rate constants, \\ complex-balanced equilibria \\ are diagonally stable \\ (in their stoichiometric classes).
\end{tabular}
& $\Rightarrow$ &
\begin{tabular}{c}
For all rate constants, \\ complex-balanced equilibria \\ are linearly stable \\ (in their stoichiometric classes).
\end{tabular} \\
\\[-1ex]
$\Downarrow$ & & $\Downarrow$  \\
\\[-1ex]
\begin{tabular}{c}
For all cycles $C$, \\
$A^C$ is diagonally D-semistable on $S^C$.
\end{tabular}
& $\Rightarrow$ &
\begin{tabular}{c}
For all cycles $C$, \\
$A^C$ is D-semistable on $S^C$.
\end{tabular}
\end{tabular}
\end{center}

\end{theorem}
\begin{proof}
The implications in the first (respectively, second) row follows immediately from Proposition~\ref{pro:lyap_on_S} (respectively, Proposition~\ref{pro:stab_implications_on_S}).

Next, we prove the implication in the right column.
Assume there exists a cycle $C$ in $G$ and a matrix $D \in \mathcal{D}_+$ such that $A^CD$ has an eigenvalue with positive real part. 
Let $x^*\in\R^n_+$ be such that $D = \diag((x^*)^{-1})$. 
Further, let $k^\varepsilon \in \R^E_+$ (with $\varepsilon>0$) be a family of rate constants as in Lemma~\ref{lem:approx_a_cycle}, and let $J^\varepsilon$ denote the corresponding Jacobian matrix. 
Since $J^\varepsilon \to A^CD$ as $\varepsilon \to 0$ and, in general, the spectrum of a matrix depends continuously on its entries, the matrix $J^\varepsilon$ has an eigenvalue with positive real part for $\varepsilon>0$ small enough.
That is, the complex-balanced equilibrium $x^*$ is not linearly stable.

Finally, we prove the implication in the left column (by contradiction).
Assume there exists a cycle $C$ in $G$ and a matrix $D \in \mathcal{D}_+$ such that for all $P \in \mathcal{D}_+$ there exists a $w \in S^C$ with
\begin{align}\label{eq:vstar_larger_0}
w^\trans(PA^CD+D(A^C)^\trans P)w>0.
\end{align}
Clearly, the term $v^\trans(PB+B^\trans P)v$ is continuous (in all arguments $v \in S$, $P \in \mathcal{D}_+$, $B \in \R^{n \times n}$), 
and hence the map $g$, defined by
\begin{align*}
g(B) = \min_{\substack{P\in\mathcal{D}_+ \colon \\ \|P\|=1}}
\;
\max_{\substack{v\in S \colon \\ \|v\|=1}}
\;
v^\trans(PB+B^\trans P)v
\end{align*}
is also continuous, since maximum and minimum are taken over compact sets. Since $S^C$ is a subspace of $S$ (and hence $w \in S$), the inequality~\eqref{eq:vstar_larger_0} implies $g(A^CD)>0$.


As above, let $x^*\in\R^n_+$ be such that $D = \diag((x^*)^{-1})$. 
Further, let $k^\varepsilon \in \R^E_+$ (with $\varepsilon>0$) be a family of rate constants as in Lemma~\ref{lem:approx_a_cycle}, and let $J^\varepsilon$ denote the corresponding Jacobian matrix. 
Since, by assumption, all complex-balanced equilibria are diagonally stable, there exists $P^\varepsilon \in \mathcal{D}_+$ such that $P^\varepsilon J^\varepsilon+(J^\varepsilon)^\trans P^\varepsilon<0$ on $S$. 
As a consequence, $g(J^\varepsilon)<0$ for all $\varepsilon>0$, and
\begin{align*}
0 < g(A^CD) = g(\lim_{\varepsilon\to0} J^\varepsilon) = \lim_{\varepsilon\to0}g(J^\varepsilon) \leq 0,
\end{align*}
a contradiction.
\end{proof}

\section{Characterization of D-stability and diagonal stability} \label{sec:characterize_restricted_stabilities}

For both D-stability and diagonal stability, an explicit characterization is available only up to dimension three. 
For arbitrary dimension, we recall the following necessary conditions (see e.g.~\cite[Section 2]{Cross1978}).

\begin{proposition} \label{pro:Dstab_diagstab_Pzeroplus_P}
Let $A \in \R^{n\times n}$. The following statements hold.
\begin{itemize}
\item[(i)] If $A$ is D-stable, then $A$ is a $P_0^+$-matrix (that is, all its signed principal minors are non-negative and at least one of each order is positive).
\item[(ii)] If $A$ is diagonally stable, then $A$ is a $P$-matrix (that is, all its signed principal minors are positive).
\end{itemize}
\end{proposition}

It is relatively easy to check that, for dimension two, equivalence holds in the previous result.

\begin{proposition} \label{prop:2by2_characterization}
Let $A \in \R^{2 \times 2}$. The following statements hold.
\begin{itemize}
\item[(i)] The matrix $A$ is D-stable if and only if it is a $P_0^+$-matrix.
\item[(ii)] The matrix $A$ is diagonally stable if and only if it is a $P$-matrix.
\end{itemize}
\end{proposition}

For dimension three, we have the following characterizations (see~\cite{Cain1976} and~\cite[Theorem 4(c)]{Cross1978}). 
Thereby, for $A \in \R^{n \times n}$, let $M_{ij}$ denote the principal minor (of order two) $a_{ii}a_{jj}-a_{ij}a_{ji}$.

\newpage

\begin{proposition} \label{prop:3by3_characterization}
Let $A \in \R^{3 \times 3}$. The following statements hold.
\begin{itemize}
\item[(i)] The matrix $A$ is D-stable if and only if
\begin{itemize}
\item[(a)] $A$ is a $P_0^+$-matrix and
\item[(b)] the three pairs $(-a_{11},M_{23})$, $(-a_{22},M_{13})$, $(-a_{33},M_{12})$ dominate $-\det A$, that is,
\begin{align*}
\left(\sqrt{-a_{11}M_{23}} + \sqrt{-a_{22}M_{13}} + \sqrt{-a_{33}M_{12}}\right)^2\geq -\det A
\end{align*}
with equality implying that at least one of the three pairs has exactly one member equal to zero.
\end{itemize}
\item[(ii)] The matrix $A$ is diagonally stable if and only if
\begin{itemize}
\item[(a)] $A$ is a $P$-matrix and
\item[(b)] there exists $y\in\R$ such that
\begin{align*}
(a_{13} \, y+a_{31})^2 - 4 \, a_{11} a_{33} \, y &< 0 \text{ and}\\
(b_1 \, y+b_2)^2 - 4 \, M_{12} M_{23} \, y &< 0,
\end{align*}
where $b_1 = a_{12}a_{23}-a_{22}a_{13}$ and $b_2 = a_{21}a_{32}-a_{22}a_{31}$.
\end{itemize}
\end{itemize}
\end{proposition}

Finally, as new results, we characterize D-(semi)stability on a linear subspace with dimension one or two.

\begin{proposition} \label{prop:Dstab_1dimS}
Let $A \in \R^{n \times n}$ and $S \subseteq \R^n$ be a linear subspace with $\im A \subseteq S$
and $\dim S = 1$.
Then, $A$ is D-semistable on $S$ if and only if $a_{ii} \leq 0$ for all $i$.
\end{proposition}
\begin{proof}
For $D \in \mathcal{D}_+$, the characteristic polynomial of $AD$ is
\[
\det (\lambda I - AD) = \lambda^{n-1}(\lambda + b)\text{ with } b = \sum_{1\leq i\leq n} (-a_{ii})d_i.
\]
Thus, $A$ is D-semistable on $S$ if and only if $b\geq 0$ for all $d_1,\ldots,d_n>0$ which is equivalent to $a_{ii} \leq 0$ for all $i$.
\end{proof}

\begin{proposition} \label{prop:Dstab_2dimS}
Let $A \in \R^{n \times n}$ and $S \subseteq \R^n$ be a linear subspace with $\im A \subseteq S$
and $\dim S = 2$.
Then, $A$ is D-stable on $S$ if and only if
\begin{itemize}
\item[(a)] $a_{ii} \leq 0$ for all $i$ and $a_{ii} < 0$ for some $i$ and
\item[(b)] $M_{ij} \geq 0$ for all $i\neq j$ and $M_{ij}>0$ for some $i\neq j$.
\end{itemize}
\end{proposition}
\begin{proof}
For $D \in \mathcal{D}_+$, the characteristic polynomial of $AD$ is
\[
\det (\lambda I - AD) = \lambda^{n-2}(\lambda^2 + b \lambda + c)
\]
with
\[
b = \sum_{1\leq i\leq n} (-a_{ii})d_i \quad \text{and} \quad c = \sum_{\substack{1\leq i,j\leq n \\ i\neq j}} M_{ij}d_id_j .
\]
Thus, $A$ is D-stable on $S$ if and only if $b>0$ and $c>0$ for all $d_1,\ldots,d_n>0$
which is equivalent to $(a)$ and $(b)$.
\end{proof}


\section{Examples} \label{sec:examples}

To illustrate our main results, Theorems~\ref{thm:main_cycle} and~\ref{thm:main_wr},
we present a series of examples.
In particular, we consider the following networks, leading to special ODE systems:
\begin{itemize}
\item irreversible three-cycle (dimension two) $\to$ trinomial ODE system
\item irreversible three-cycle (dimension three, stoichiometric subspace of dimension two) $\to$ binomial ODE system
\item irreversible four-cycle (dimension three) $\to$ binomial ODE system 
\item reversible chain (arbitrary dimension)
\item S-system (arbitrary dimension) $\to$ binomial ODE system
\end{itemize}
For the two- and three-dimensional examples,
we use the characterizations of D-stability and diagonal stability
given in Propositions~\ref{prop:2by2_characterization},~\ref{prop:3by3_characterization},
and~\ref{prop:Dstab_2dimS}.
For the examples with arbitrary dimension, we use Proposition~\ref{prop:Dstab_1dimS}.


\begin{example}\label{ex:n2_m3_general}
Consider the generalized mass-action system
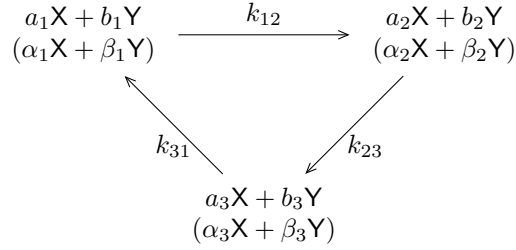
\begin{center}
\begin{tikzpicture}[scale=.8]
\node (0) at (0,0)  {$\begin{array}{c} a_1 \mathsf{X} + b_1 \mathsf{Y} \\ (\alpha_1 \mathsf{X} + \beta_1 \mathsf{Y})\end{array}$};
\node (X) at (6,0)  {$\begin{array}{c} a_2 \mathsf{X} + b_2 \mathsf{Y} \\ (\alpha_2 \mathsf{X} + \beta_2 \mathsf{Y})\end{array}$};
\node (Y) at (3,-3) {$\begin{array}{c} a_3 \mathsf{X} + b_3 \mathsf{Y} \\ (\alpha_3 \mathsf{X} + \beta_3 \mathsf{Y})\end{array}$};
\draw[arrows={-angle 45}] (0) to node [above] {$k_{12}$} (X);
\draw[arrows={-angle 45}] (X) to node [below] {$\phantom{.} \;\; k_{23}$} (Y);
\draw[arrows={-angle 45}] (Y) to node [below] {$k_{31}$} (0);
\end{tikzpicture}
\end{center}
and the resulting ODE system
\begin{align*}
\dd{x}{t} &= k_{12}(a_2-a_1)x^{\alpha_1} y^{\beta_1} + k_{23}(a_3-a_2)x^{\alpha_2} y^{\beta_2} + k_{31}(a_1-a_3)x^{\alpha_3} y^{\beta_3}, \\
\dd{y}{t} &= k_{12}(b_2-b_1)x^{\alpha_1} y^{\beta_1} + k_{23}(b_3-b_2)x^{\alpha_2} y^{\beta_2} + k_{31}(b_1-b_3)x^{\alpha_3} y^{\beta_3}.
\end{align*}

The matrix $YA_{k=1}\widetilde Y^\trans$ is given by
\begin{align*}
(YA_{k=1}\widetilde Y^\trans)_{11} &= \alpha_1(a_2-a_1) + \alpha_2(a_3-a_2) + \alpha_3(a_1-a_3), \\
(YA_{k=1}\widetilde Y^\trans)_{12} &= \beta_1(a_2-a_1) + \beta_2(a_3-a_2) + \beta_3(a_1-a_3), \\
(YA_{k=1}\widetilde Y^\trans)_{21} &= \alpha_1(b_2-b_1) + \alpha_2(b_3-b_2) + \alpha_3(b_1-b_3), \\
(YA_{k=1}\widetilde Y^\trans)_{22} &= \beta_1(b_2-b_1) + \beta_2(b_3-b_2) + \beta_3(b_1-b_3),
\end{align*}
and
\begin{align*}
\det(Y A_{k = 1} \widetilde Y^\trans) =
\underbrace{\det \begin{bmatrix}a_2-a_1 & b_2-b_1 \\ a_3-a_2 & b_3-b_2 \end{bmatrix}}_{d_{ab}}
\cdot
\underbrace{\det\begin{bmatrix}\alpha_2-\alpha_1 & \beta_2-\beta_1 \\ \alpha_3-\alpha_2 & \beta_3-\beta_2 \end{bmatrix}}_{d_{\alpha\beta}} .
\end{align*}
Now, assume $d_{ab} \neq 0$, that is, $\dim S = 2$.
By Theorem~\ref{thm:main_cycle} and Proposition~\ref{prop:2by2_characterization}(i),
the unique complex-balanced equilibrium is linearly stable for all $k$ {\em if and only if}
\begin{align*}
&d_{ab} \, d_{\alpha\beta} > 0,\\
&\alpha_1(a_2-a_1) + \alpha_2(a_3-a_2) + \alpha_3(a_1-a_3) \leq 0,\\
&\beta_1(b_2-b_1) + \beta_2(b_3-b_2) + \beta_3(b_1-b_3) \leq 0,
\end{align*}
and the latter two are not both zero.
\end{example}

\begin{example}\label{ex:ivanova}
Consider the generalized mass-action system
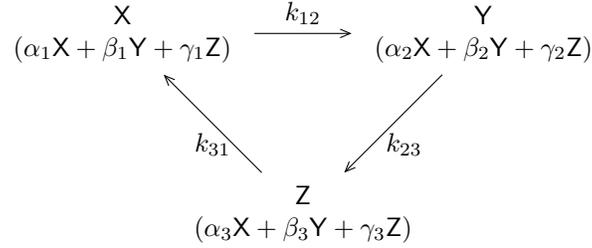
\begin{center}
\begin{tikzpicture}[scale=.8]
\node (X) at (0,0)  {$\begin{array}{c} \mathsf{X} \\ (\alpha_1 \mathsf{X} + \beta_1 \mathsf{Y} + \gamma_1 \mathsf{Z})\end{array}$};
\node (Y) at (6,0)  {$\begin{array}{c} \mathsf{Y} \\ (\alpha_2 \mathsf{X} + \beta_2 \mathsf{Y} + \gamma_2 \mathsf{Z})\end{array}$};
\node (Z) at (3,-3) {$\begin{array}{c} \mathsf{Z} \\ (\alpha_3 \mathsf{X} + \beta_3 \mathsf{Y} + \gamma_3 \mathsf{Z})\end{array}$};
\draw[arrows={-angle 45}] (X) to node [above] {$k_{12}$} (Y);
\draw[arrows={-angle 45}] (Y) to node [below] {$\phantom{.} \;\; k_{23}$} (Z);
\draw[arrows={-angle 45}] (Z) to node [below] {$k_{31}$} (X);
\end{tikzpicture}
\end{center}
and the resulting ODE system
\begin{align*}
\dd{x}{t} &= k_{31} x^{\alpha_3} y^{\beta_3} z^{\gamma_3} - k_{12} x^{\alpha_1} y^{\beta_1} z^{\gamma_1}, \\
\dd{y}{t} &= k_{12} x^{\alpha_1} y^{\beta_1} z^{\gamma_1} - k_{23} x^{\alpha_2} y^{\beta_2} z^{\gamma_2}, \\
\dd{z}{t} &= k_{23} x^{\alpha_2} y^{\beta_2} z^{\gamma_2} - k_{31} x^{\alpha_3} y^{\beta_3} z^{\gamma_3}.
\end{align*}
Clearly, $\dim S = 2$.
The matrix $YA_{k=1}\widetilde Y^\trans$ is given by
\begin{align*}
\begin{bmatrix}
\alpha_3-\alpha_1 & \beta_3-\beta_1 & \gamma_3-\gamma_1 \\ \alpha_1-\alpha_2 & \beta_1-\beta_2 & \gamma_1-\gamma_2 \\ \alpha_2-\alpha_3 & \beta_2-\beta_3 & \gamma_2-\gamma_3
\end{bmatrix}.
\end{align*}
By Theorem~\ref{thm:main_cycle} and Proposition~\ref{prop:Dstab_2dimS},
all the complex-balanced equilibria are linearly stable if and only if
\begin{align*}
\alpha_3 \leq \alpha_1, \beta_1 \leq \beta_2, \gamma_2 \leq \gamma_3
\end{align*}
with at least one of the three inequalities satisfied strictly and
\begin{align*}
\det\begin{bmatrix}
\alpha_3-\alpha_1 & \beta_3-\beta_1 \\ \alpha_1-\alpha_2 & \beta_1-\beta_2
\end{bmatrix}&\geq 0,\\
\det\begin{bmatrix}
\alpha_3-\alpha_1 & \gamma_3-\gamma_1 \\ \alpha_2-\alpha_3 & \gamma_2-\gamma_3
\end{bmatrix}&\geq 0,\\
\det\begin{bmatrix}
\beta_1-\beta_2 & \gamma_1-\gamma_2 \\ \beta_2-\beta_3 & \gamma_2-\gamma_3
\end{bmatrix}&\geq 0
\end{align*}
with at least one of the three inequalities satisfied strictly.
\end{example}

\begin{example} \label{ex:n3_m4_general}
Consider the generalized mass-action system
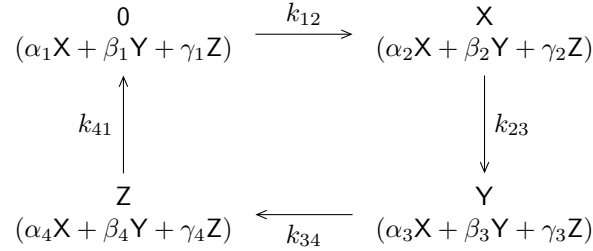
\begin{center}
\begin{tikzpicture}[scale=.8]
\node (0) at (0,0) {$\begin{array}{c} \mathsf{0} \\ (\alpha_1 \mathsf{X} + \beta_1 \mathsf{Y} + \gamma_1 \mathsf{Z})\end{array}$};
\node (X) at (6,0) {$\begin{array}{c} \mathsf{X} \\ (\alpha_2 \mathsf{X} + \beta_2 \mathsf{Y} + \gamma_2 \mathsf{Z})\end{array}$};
\node (Y) at (6,-3) {$\begin{array}{c} \mathsf{Y} \\ (\alpha_3 \mathsf{X} + \beta_3 \mathsf{Y} + \gamma_3 \mathsf{Z})\end{array}$};
\node (Z) at (0,-3) {$\begin{array}{c} \mathsf{Z} \\ (\alpha_4 \mathsf{X} + \beta_4 \mathsf{Y} + \gamma_4 \mathsf{Z})\end{array}$};
\draw[arrows={-angle 45}] (0) to node [above] {$k_{12}$} (X);
\draw[arrows={-angle 45}] (X) to node [right] {$k_{23}$} (Y);
\draw[arrows={-angle 45}] (Y) to node [below] {$k_{34}$} (Z);
\draw[arrows={-angle 45}] (Z) to node [left] {$k_{41}$} (0);
\end{tikzpicture}
\end{center}
and the resulting ODE system
\begin{align*}
\dd{x}{t} &= k_{12} x^{\alpha_1} y^{\beta_1} z^{\gamma_1} - k_{23} x^{\alpha_2} y^{\beta_2} z^{\gamma_2}, \\
\dd{y}{t} &= k_{23} x^{\alpha_2} y^{\beta_2} z^{\gamma_2} - k_{34} x^{\alpha_3} y^{\beta_3} z^{\gamma_3}, \\
\dd{z}{t} &= k_{34} x^{\alpha_3} y^{\beta_3} z^{\gamma_3} - k_{41} x^{\alpha_4} y^{\beta_4} z^{\gamma_4}.
\end{align*}

Clearly, $\dim S = 3$.
The matrix $YA_{k=1}\widetilde Y^\trans$ is given by
\begin{align*}
\begin{bmatrix}
\alpha_1-\alpha_2 & \beta_1-\beta_2 & \gamma_1-\gamma_2 \\ \alpha_2-\alpha_3 & \beta_2-\beta_3 & \gamma_2-\gamma_3 \\ \alpha_3-\alpha_4 & \beta_3-\beta_4 & \gamma_3-\gamma_4
\end{bmatrix}.
\end{align*}
For simplicity, we consider particular kinetic orders
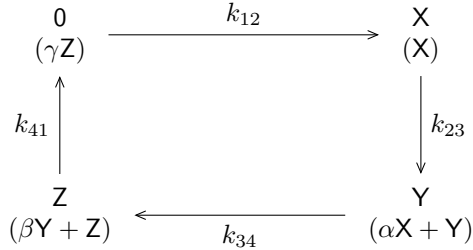
\begin{center}
\begin{tikzpicture}[scale=.8]
\node (0) at (0,0) {$\begin{array}{c} \mathsf{0} \\ (\gamma \mathsf{Z})\end{array}$};
\node (X) at (6,0) {$\begin{array}{c} \mathsf{X} \\ (\mathsf{X})\end{array}$};
\node (Y) at (6,-3) {$\begin{array}{c} \mathsf{Y} \\ (\alpha \mathsf{X} + \mathsf{Y})\end{array}$};
\node (Z) at (0,-3) {$\begin{array}{c} \mathsf{Z} \\ (\beta \mathsf{Y} + \mathsf{Z})\end{array}$};
\draw[arrows={-angle 45}] (0) to node [above] {$k_{12}$} (X);
\draw[arrows={-angle 45}] (X) to node [right] {$k_{23}$} (Y);
\draw[arrows={-angle 45}] (Y) to node [below] {$k_{34}$} (Z);
\draw[arrows={-angle 45}] (Z) to node [left] {$k_{41}$} (0);
\end{tikzpicture}
\end{center}
and the resulting matrix
\begin{align*}
YA_{k=1}\widetilde Y^\trans = \begin{bmatrix*}
-1 & 0 & \gamma \\ 1-\alpha & -1 & 0 \\ \alpha & 1-\beta & -1
\end{bmatrix*}.
\end{align*}

In the following table, we choose particular values for the kinetic orders $\alpha$, $\beta$, and $\gamma$,
leading to different situations regarding stability, D-stability, diagonal stability, and being a $P_0^+$-matrix.
\begin{center}
\begin{tabular}{rrr|l}
$\alpha$ & $\beta$ & $\gamma$ & $YA_{k=1}\widetilde Y^\trans$\\
\hline
\hline
$0$ &  $0$ &  $0$ & \text{diagonally stable} \\
$5$ &  $0$ & $-3$ & \text{D-stable, but not diagonally stable} \\
$3$ &  $4$ & $-4$ & \text{stable $P_0^+$-matrix, but not D-stable} \\
$2$ & $-2$ &  $1$ & \text{stable, but not $P_0^+$-matrix} \\
$0$ & $-2$ & $-3$ & \text{unstable $P_0^+$-matrix} 
\end{tabular}
\end{center}
In the first and second case,
for all rate constants, 
there exists a complex-balanced equilibrium
which is linearly stable by Theorem~\ref{thm:main_cycle} and Proposition~\ref{prop:3by3_characterization}(i)
and unique by Theorem~\ref{thm:main_injectivity}.
In the third and fourth case,
the unique complex-balanced equilibrium is linearly stable for some rate constants, but not for all.
\end{example}

\begin{example} \label{ex:rev_chain}
Consider the generalized chemical reaction network
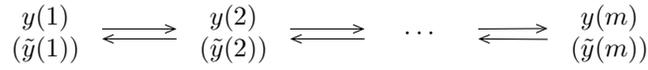
\begin{center}
\begin{tikzpicture}
\node (Z1) at (0,0) {$\begin{array}{c} y(1) \\ (\tilde y(1))\end{array}$};
\node (Z2) at (2.5,0) {$\begin{array}{c} y(2) \\ (\tilde y(2))\end{array}$};
\node (ZM) at (5,0) {$\cdots$};
\node (Z3) at (5,0) {\phantom{$\begin{array}{c} y(2) \\ (\tilde y(2))\end{array}$}};
\node (Z4) at (7.5,0) {$\begin{array}{c} y(m) \\ (\tilde y(m))\end{array}$};
\draw[arrows={-angle 45}] (Z1.5) to node [above] {} (Z2.175);
\draw[arrows={-angle 45}] (Z2.-175) to node [below] {} (Z1.-5);
\draw[arrows={-angle 45}] (Z2.5) to node [above] {} (Z3.175);
\draw[arrows={-angle 45}] (Z3.-175) to node [below] {} (Z2.-5);
\draw[arrows={-angle 45}] (Z3.5) to node [above] {} (Z4.175);
\draw[arrows={-angle 45}] (Z4.-175) to node [below] {} (Z3.-5);
\end{tikzpicture}
\end{center}
where the graph $G$ is a reversible chain.
In particular, all cycles of $G$ correspond to reversible reactions.
For every cycle $i \rightleftarrows i+1$ (denoted by $C$), we have
\[
Y A_{k=1}^C \widetilde Y^\trans = -(y(i+1)-y(i))(\tilde{y}(i+1)-\tilde{y}(i))^\trans ,
\]
a dyadic product. 
By Proposition~\ref{prop:Dstab_1dimS},
this (rank-one) matrix is D-semistable on $S^C = \im (y(i+1)-y(i))$ if and only if all diagonal entries are non-positive. 
By Theorem~\ref{thm:main_wr},
if there is a cycle $i \rightleftarrows i+1$ and an index (species) $s \in \{1,\ldots,n\}$
with 
\[
-(y(i+1)-y(i))_s(\tilde{y}(i+1)-\tilde{y}(i))_s > 0 ,
\]
then there is a complex-balanced equilibrium (for some rate constants) that is not linearly stable.
\end{example}

\begin{example}
For the real matrices $G,H \in \R^{n \times n}$ and the positive vectors $\alpha, \beta \in \R^n_+$, the ODE system
\begin{align*}
\dot{x}_1 &= \alpha_1 \, x_1^{g_{11}}\cdots x_n^{g_{1n}} - \beta_1 \, x_1^{h_{11}}\cdots x_n^{h_{1n}} \\
          & \;\, \vdots \\
\dot{x}_n &= \alpha_n \, x_1^{g_{n1}}\cdots x_n^{g_{nn}} - \beta_n \, x_1^{h_{n1}}\cdots x_n^{h_{nn}} \\
\end{align*}
on $\R^n_+$ is called an \emph{S-system}~\cite{Savageau1969,Voit2013}. 
This ODE is associated to the generalized mass-action system
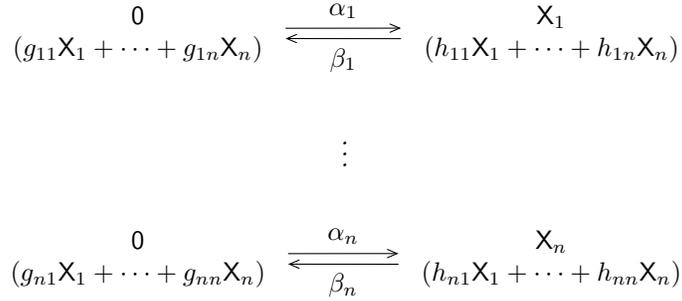
\begin{center}
\begin{tikzpicture}
\node (Z1) at (0,0) {$\begin{array}{c} \mathsf{0} \\ (g_{11}\mathsf{X}_1 + \cdots + g_{1n}\mathsf{X}_n)\end{array}$};
\node (X1) at (5.5,0) {$\begin{array}{c} \mathsf{X}_1 \\ (h_{11}\mathsf{X}_1 + \cdots + h_{1n}\mathsf{X}_n)\end{array}$};
\draw[arrows={-angle 45}] (Z1.2) to node [above] {$\alpha_1$} (X1.178);
\draw[arrows={-angle 45}] (X1.-178) to node [below] {$\beta_1$} (Z1.-2);
\node (dotdotdot) at (2.75,-1.5) {$\vdots$};
\node (Zn) at (0,-3) {$\begin{array}{c} \mathsf{0} \\ (g_{n1}\mathsf{X}_1 + \cdots + g_{nn}\mathsf{X}_n)\end{array}$};
\node (Xn) at (5.5,-3) {$\begin{array}{c} \mathsf{X}_n \\ (h_{n1}\mathsf{X}_1 + \cdots + h_{nn}\mathsf{X}_n)\end{array}$};
\draw[arrows={-angle 45}] (Zn.2) to node [above] {$\alpha_n$} (Xn.178);
\draw[arrows={-angle 45}] (Xn.-178) to node [below] {$\beta_n$} (Zn.-2);
\end{tikzpicture}
\end{center}

For every cycle $C$ (corresponding to a reversible reaction $\mathsf{0} \rightleftarrows \mathsf{X}_i$),
the matrix $YA_{k=1}^C\widetilde{Y}^\trans$ has $(g_{\cdot i}-h_{\cdot i})^\trans$ as its $i$-th row,
and all other rows are zero. 
By Proposition~\ref{prop:Dstab_1dimS},
this (rank-one) matrix is D-semistable on $S^C = \im (e_i)$ if and only if the diagonal entry $g_{ii}-h_{ii}$ is non-positive. 
By Theorem~\ref{thm:main_wr},
if there is a reversible reaction $\mathsf{0} \rightleftarrows \mathsf{X}_i$ such that $g_{ii} > h_{ii}$,
then there is a complex-balanced equilibrium (for some rate constants) that is not linearly stable.
\end{example}


\paragraph{Acknowledgments.}
BB and SM were supported by the Austrian Science Fund (FWF), project 28406. 
GR was supported by the FWF, project 27229.


\bibliographystyle{abbrv}
\bibliography{linearstability}

\end{document}